\documentclass{article}
\usepackage{amssymb,amsmath}
\newtheorem{Theorem}{Theorem}[section]
\newtheorem{Lemma}[Theorem]{Lemma}
\newtheorem{Corollary}[Theorem]{Corollary}
\newenvironment{proof}[1][Proof]{\textbf{#1.} }{\
\rule{0.5em}{0.5em}}

\newcommand{\bC}{\mathbb{C}}
\newcommand{\ord}{\mbox{\rm ord}}
\newcommand{\ini}{\mbox{\rm in}}
\newcommand{\unit}{\hbox{$\,o\!\!\!\!$\lower0.06ex\hbox{---}\,}}
\title{Initial Newton polynomial of the discriminant}
\author{Beata Gryszka, Janusz Gwo\'zdziewicz and Adam Parusi\'nski}

\begin{document}
\maketitle
\footnotetext{
      \begin{minipage}[t]{4.2in}{\small
       2020 {\it Mathematics Subject Classification:\/} Primary 32S15;
       Secondary 32S45.\\
       Key words and phrases:  discriminant curve, Newton diagram, equisingularity}
       \end{minipage}}

\begin{abstract}
Let $(f,g)\colon (\mathbb{C}^2,0)\longrightarrow (\mathbb{C}^2,0)$ be a holomorphic 
mapping with an isolated zero. 
We show that the initial Newton polynomial of its discriminant 
is determined, up to rescalling variables, by the ideals $(f)$ and $(g)$.
\end{abstract}

\section{Introduction}
\

\medskip
%\textit{Notation.}
Let  $\mathbb{R}_{\geq 0}$ $(\mathbb{Z}_{\geq 0})$ be the set of all non-negative real (integer) numbers. For a~power series $f=\sum_{(i,j)\in \mathbb{Z}_{\geq 0}^2}a_{i,j}x^{i}y^{j}\in\mathbb{C}[[x,y]]$ we define its {\it Newton diagram} $\Delta(f)$ as the convex hull of the union $\bigcup_{a_{i,j}\ne 0}((i,j)+\mathbb{R}_{\geq 0}^2)$. If  $S$ is the union of all compact edges of $\Delta(f)$, then the polynomial $f|_S:=\sum_{(i,j)\in S}a_{i,j}x^iy^j$ is called the {\it initial  Newton polynomial} of  $f$. We say that power series $f_1,f_2\in\mathbb{C}[[x,y]]$ are 
\textit{equal up to rescalling variables} if $f_2(x,y)=f_1(ax,by)$ for some nonzero constants $a$, $b$. 

Let $\phi=(f,g)\colon (\mathbb{C}^2,0)\rightarrow (\mathbb{C}^2,0)$ be the germ of a  holomorphic mapping with an isolated zero. To any germ $\xi$ of an analytic curve in $(\mathbb{C}^2,0)$ one associates its \textit{direct image} $\phi_{*}(\xi)$, 
see for example \cite{Casas1,Casas2}.  
The direct image of $\xi$ by $\phi$ is an analytic curve germ in the target space  uniquely determined by the following two properties:

\begin{itemize}
\item[(i)] if $\xi\subset (\mathbb{C}^2,0)$ is an irreducible curve 
then $\phi_{*}(\xi)$ is the curve of equation $H^d=0$, where 
$H=0$ is a reduced equation of the curve $\phi(\xi)$ in the target space and 
$d$ is the topological degree of the restriction $\phi|_{\xi}:\xi\to \phi(\xi)$.
\item[(ii)] if $h=h_1\cdots h_s$ is a factorization of a power series $h$ to the product of  irreducible factors in $\mathbb{C}\{x,y\}$, then $\phi_{*}(\{h=0\})$ 
is the curve $H_1\cdots H_s=0$, where the curves $H_i=0$ are the direct images of the
branches $h_i=0$ for $i=1,\dots,s$.
\end{itemize}

The direct image satisfies the projection formula: 
If $H=0$ is the direct image of $h=0$ then for any analytic curve $w=0$ in the target space 
we have the equality of intersection multiplicities $i_0(w\circ\phi,h)=i_0(w,H)$.

Let $H=0$ be the direct image of $h=0$. Suppose that the Newton diagram of $H$ has $r$ edges which are not contained in the coordinate axes. Then $H$ can be written as a product $H_1\cdots H_r$,
where the Newton diagram of each $H_i$ is elementary, i.e.  has exactly one edge not contained in the coordinate axes.  A factorization $H=H_1\cdots H_r$ induces a factorization $h=h_1\cdots h_r$
such that $H_i=0$ is the direct image of $h_i=0$ for $i=1,\dots,r$. 

We will call such a factorization of $h$ the \textit{Hironaka factorization}, because 
it follows from the projection formula applied to the coordinate axes in the target space that for every 
ireducible factor $p$ of $h_i$ the \textit{Hironaka quotient} $i_0(g,p)/i_0(f,p)$
is the inclination of the edge of $\Delta(H_i)$. Moreover the intersection multiplicities $i_0(f,h_i)$, $i_0(g,h_i)$ determine and are determined by the edges of the Newton diagram $\Delta(H)$. 

In this article we apply this general construction to a more specific situation.  We call  
${\rm Jac}(\phi)=\tfrac{\partial f}{\partial x}\tfrac{\partial g}{\partial y}-
\tfrac{\partial f}{\partial y}\tfrac{\partial g}{\partial x}=0$ the \textit{Jacobian curve} of $\phi$ and 
the direct image of the Jacobian curve is called the \textit{discriminant curve}.  
The Newton diagram of the discriminant, denoted $Q(f,g)$, is called the 
{\it Jacobian Newton diagram} of $(f,g)$ and it was introduced by Tessier in \cite{Tes1,Tes2}. 

The Jacobian curve in the case of $f=0$ smooth and transverse to $g=0$ 
is called the generic polar curve of $g$.  In this case the Hironaka quotients 
$i_0(g,h) / i_0(f,h)$ where $h$ is an irreducible factor of ${\rm Jac}(\phi)$, are called the \emph{polar quotients}. 

The polar case has been widely studied. 
The polar quotients are invariants of singularity as shown in \cite{KL}.
Teissier \cite{Tes1,Tes2} proved that the Jacobian Newton diagram in polar case 
is also a singularity  invariant. Merle \cite{Merle} described $Q(f,g)$ for $g$ irreducible. 
Eggers \cite{Egg} gave a description of $Q(f,g)$ for any $g$.  

The case where both curves $f=0$ and $g=0$ are singular is more complicated. 
Maugendre \cite{Ma} characterized Hironaka quotients of the Jacobian curve. 
%called in this case Jacobian quotients. 
Michel  \cite{Michel} gave formulas for $i_0(f,h_i)$, $i_0(g,h_i)$ using topological tools, thus determining the Jacobian Newton diagram $Q(f,g)$. Michel's result was reproved in \cite{Gw} 
where it is shown that $Q(f,g)$ depends only on the equisingularity type of pairs 
of curves $f=0$, $g=0$.

 Recall that pairs of curves $f=0$, $g=0$ and $\tilde f=0$, $\tilde g=0$ are {\it equisingular} if there exist factorisations $f=h_1\cdots h_s$, $g=h_{s+1}\cdots h_r$, $\tilde f=\tilde h_1\cdots \tilde h_{\tilde s}$, $\tilde g=\tilde h_{\tilde{s}+1}\cdots \tilde h_{\tilde r}$ into the product of   irreducible factors in $\mathbb{C}\{x,y\}$ 
such that
\begin{itemize}
\item $s=\tilde{s}$, $r=\tilde{r}$, 
\item for $i=1,\dots,r$, the semigroups 
$\Gamma(h_i):=\{i_0(h_i,w):  w\notin(h_i)\}$ and 
$\Gamma(\tilde h_i):=\{i_0(\tilde h_i,w):w\notin(\tilde h_i)\}$ 
are equal,
\item $i_0(h_i,h_j)=i_0(\tilde h_i,\tilde h_j)$ for $1\leq i<j\leq r$.
\end{itemize}

In this paper we go a step further. Our main result is.

\begin{Theorem} \label{T:main}
Let $f,g,u',u''\in\mathbb{C}\{x,y\}$ be convergent power series vanishing at zero 
such that $f$ and $g$ are coprime and let $\tilde f=(1+u')f$, $\tilde g=(1+u'')g$.
Then the initial Newton polynomials of discriminants of mappings  
$(f,g):(\mathbb{C}^2,0)\to(\mathbb{C}^2,0)$ and 
$(\tilde f,\tilde g):(\mathbb{C}^2,0)\to(\mathbb{C}^2,0)$ are 
equal. 
\end{Theorem}
%--------------------------------------------------

Note that a factorization of the initial Newton polynomial of the discriminant $D$ 
of $\phi=(f,g)$ into irreducible factors, gives a factorization $D=D_1\cdots D_s$ such that  the initial Newton polynomial of each factor $D_i$ is a power of an irreducible polynomial. 
This induces a factorization of ${\rm Jac}(\phi)$  that is usually more subtle than the Hironaka factorization.

The structure of the paper is as follows.  In Section \ref{S:Factorization} we study the initial weighted form of a nonzero power series.  We establish in Lemma \ref{L:I} a relation between the intersection multiplicities of this power series with certain test polynomials and the multiplicities of the irreducible factors of its initial weighted form. In Corollary \ref{C:Milnor}, thanks to Casas' formula, that we recall for the reader convenience, these multiplicities are expressed in terms of classical equisingularity invariants (Milnor numbers). This allows us in Section \ref{S:Key lemma} 
to show the key lemma, Lemma \ref{L:Key}.  It states that some specific plane curve singularities, constructed in terms of $f$, $g$, $\tilde f$, $ \tilde g$, and the test polynomials, are 
equisingular.  The key lemma follows from a classical  Zariski equisingularity criterion of
\cite{ZarSEI}. The main result, Theorem \ref{T:main}, is proven Section \ref{S:proof}.  
Finally we give several corollaries of the main result.  This includes a characterization of the  atypical values of the pencils $f^k=tg^l$ and of the asymptotic critical values of the meromorphic functions $f^k/g^l$.  We also propose an alternative proof of \cite[Theorem 6.6]{GaGwLe} by an argument similar to the proof of the main result.
%--------------------------------------------------

\section{Factorization of the initial weighted form}\label{S:Factorization}
Let $D(u,v)$ be a nonzero complex power series and 
let  $w=(k,l)$ be a weight vector, where $k$, $l$ are coprime positive integers. Then $D$ can be written as the sum of 
quasi-homogeneous polynomials 
$D=D_m+D_{m+1}+{}\cdots$, where $D_m\neq 0$ and $\deg_w D_i=i$ for $i\geq m$. 
Write $D_m$ as a product

\begin{equation} \label{E:w}
   D_m(u,v)=Cu^{\nu_{0}}v^{\nu_{n+1}}\prod_{i=1}^n(v^k-t_i u^l)^{\nu_i},
\end{equation}
where $t_i\neq0$ and $t_i\neq t_j$ for $i\neq j$. 

The aim of this section is to express the multiplicities  
$\nu_i$ for $1\leq i\leq n$ 
by the intersection multiplicities of $D$ with certain polynomials.

\begin{Lemma} \label{L:I}
Let $H_t=( v^k-t u^l)^N-u^{l(N+1)}$. 
Then for a sufficiently large integer~$N$ 
and for every $t\in \mathbb C^*$ such that $t\neq t_i$ for $1\leq i\leq n$, one has
$$ \nu_jkl=i_0(D,H_{t_j})-i_0(D,H_t). $$
\end{Lemma}
%--------------------------------------------------

\begin{proof}
For any power series $F(u,v)$ denote by $\ini_w F$ its weighted initial form  
with respect to the weight vector $w$. By the quasi-homogeneous version of Hensel's lemma 
(see e.g. \cite[Lemma~A.1]{Artal}) the power series $D$ factors to a product $PQ$, where
$\ini_w P=(v^k- t_j  u^l)^{\nu_j}$ and 
$\ini_w Q$ is not divisible by $v^k- t_j  u^l$.

\smallskip
Take any $N>\nu_jkl$.
In order to compute the intersection multiplicity $i_0(P,H_{t_j})$ 
we will use the classical Zeuten's Rule 
$ i_0(P,H_{t_j})=\sum \ord\, P(u,\alpha_i(u))$,
where the sum runs over all Newton-Puiseux roots $v=\alpha_i(u)$ of $H_{t_j}(u,v)=0$. 

\smallskip
Solving the equation $(v^k -t_j u^l)^N-u^{l(N+1)}=0$ with respect to $v$ we get 
$Nk$ solutions $v=\alpha_i(u)$ 
of the form $v=\omega\sqrt[k]{t_j}u^{l/k}+\mbox{\it higher order terms}$, 
where $\omega^k=1$ and $\sqrt[k]{t_j}$ is a fixed root of a polynomial $Y^k-t_j$.

Writing $P(u,v)$ as a power series 
$(v^k- t_j u^l)^{\nu_j}+\sum_{ak+bl> \nu_j kl} c_{a,b}u^av^b$
and substituting $v=\alpha_i(u)$ we get
$$ P(u,\alpha_i(u))= \epsilon u^{\nu_j l(N+1)/N}+\sum_{ak+bl> \nu_jkl} c_{a,b}u^a\alpha_i(u)^b,
$$
where $\epsilon^N=1$.

\medskip
The term $\epsilon u^{\nu_j l(N+1)/N}$ on the right-hand side of the above formula  
has the smallest order.
Indeed, by the inequalities $ak+bl\geq \nu_j kl+1$, $N>\nu_j kl$ 
we get 
$\ord\, u^a\alpha_i(u)^b = a+b(l/k) = (ak+bl)/k \geq 
\nu_j l(\nu_j kl+1)/(\nu_j kl) > \nu_jl(N+1)/N$. 
We have shown that $\ord\, P(u,\alpha_i(u))=\nu_jl(N+1)/N$ for every Newton-Puiseux root 
$v=\alpha_i(u)$ of $H_{t_j}(u,v)=0$. Hence $i_0(P,H_{t_j})=(N+1)\nu_jkl$.

\medskip
Computing the intersection multiplicity $i_0(Q,H_{t_j})$ is simpler.  Since the polynomials 
$\ini_w Q$ and $\ini_w H_{t_j}=(v^k -t_j u^l)^N$ are coprime,  we have 
$\ord\, Q(u,\alpha_i(u))=\ord\, \ini_w Q(u,\omega\sqrt[k]{t_j}u^{l/k})=(1/k)\deg_w(\ini_w Q)$
for any Newton-Puiseux root~$\alpha_i(u)$ of $H_{t_j}(u,v)=0$. 
By Zeuten's Rule we get $i_0(Q,H_{t_j}) = N\deg_w(\ini_w Q)$. 

Analogously we obtain $i_0(Q,H_t) = N\deg_w(\ini_w Q)$ and 
$i_0(P,H_t) = N\deg_w(\ini_w P) = N\nu_jkl$.

Finally 
$i_0(D,H_{t_j})-i_0(D,H_t) = (i_0(P,H_{t_j})+i_0(Q,H_{t_j}))-(i_0(P,H_t)+i_0(Q,H_t)) =\nu_jkl$ 
which ends the proof. 
\end{proof}
%--------------------------------------------------

\section{Casas' formula}\label{S:Cas}
Consider the following result of Casas-Alvero, which provides  a very useful formula.

\begin{Theorem}[\cite{Casas1} Theorem~3.2] \label{Th:Casas}
Let $(f,g):(\mathbb{C}^2,0)\to(\mathbb{C}^2,0)$ 
be the germ of a holomorphic mapping such that
$(f,g)^{-1}(0,0)=\{(0,0)\}$.
Let $D(u,v)=0$ be the discriminant 
 of $(f,g)$. Take any curve germ $H(u,v)=0$ and let $h(x,y)=H(f(x,y),g(x,y))$. Then 
$$\mu (h)-1=i_0(f,g)[\mu (H)-1]+i_0(D,H), $$ 
where $\mu (h)$   denotes the Milnor number of the curve $h=0$ at zero. 
\end{Theorem}
%--------------------------------------------------

From the above theorem we obtain a crucial corollary, which is used in the proof of the main result of this article.

\begin{Corollary} \label{C:Milnor}
Let $(f,g):(\mathbb{C}^2,0)\to(\mathbb{C}^2,0)$ be the germ of a holomorphic mapping 
such that $(f,g)^{-1}(0,0)=\{(0,0)\}$.  
Let $D(u,v)=0$ be the discriminant curve of $(f,g)$ and $h_t=(g^k-tf^l)^N-f^{l(N+1)}$ for $N>1$. Then, under the notation of (\ref{E:w}), for $N\gg1$ and $t\in \mathbb C^*$ different from  $t_1,\dots,t_n$,  we have $\nu_j kl=\mu (h_{t_j})-\mu (h_t)$.
\end{Corollary}

\begin{proof} It is enough to apply Casas' formula to $H_{t_j}$ and $H_{t}$ as defined in Lemma~\ref{L:I}. 
\end{proof}
%------------------------------------------

\section{Key lemma and the proof of the main result}\label{S:Key lemma}
Before we present the proof of the main result, we provide the following lemma.  This lemma is interesting not only because of its applications, but also because of the methods, which are used in its proof. At the end of this paper we show another application of these methods, presenting a new proof of some known fact.
\begin{Lemma}[Key Lemma]\label{L:Key}
Let $f,g,u',u''\in\mathbb{C}\{x,y\}$ be convergent power series vanishing at zero 
such that $f$ and $g$ are coprime and let $\tilde f=(1+u')f$, $\tilde g=(1+u'')g$.
Then for sufficiently large integer $N$ the curves 
$(g-f)^N-f^{N+1}=0$ and $(\tilde g-\tilde f)^N-\tilde f^{N+1}=0$ are equisingular.
\end{Lemma}
\begin{proof}
By Zariski \cite{ZarSEI} two plane curve singularities (at the origin) are equivalent if there is a bijection between their branches that makes both their proper and total transforms by the blowing-up of the origin equivalent. In order to apply this criterion, we construct a common resolution of $h:=(g-f)^N-f^{N+1}=0$ and $\tilde h:=(\tilde g-\tilde f)^N-\tilde f^{N+1}=0$ with locally equivalent singularities (for both the proper and the total transforms). Then it follows that $h$ and $\tilde h$ are equisingular.

Let $R:M\to(\mathbb{C}^2,0)$ be a  resolution of singularities of the curve $fg(g-f)(\tilde g-\tilde f)=0$. The total transform $R^{-1}(\{fg(g-f)(\tilde g-\tilde f)=0\})$ 
can be written as the union of irreducible components
$E_1\cup\dots\cup E_n\cup E_{n+1}\cup\dots\cup E_m$,
where $E=E_1\cup\dots\cup E_n$ is the exceptional divisor $R^{-1}(0)$ 
and $E_{n+1}$, \dots, $E_m$  are the components of the proper transform  of the curve $fg(g-f)(\tilde g-\tilde f)=0$.
By abuse of notation we will use the same symbols for germs of functions on 
$(\mathbb{C}^2,0)$ and for their pull-backs to $M$.

For $i=1,\dots,m$ we denote the orders of $f,g,g-f, \tilde f, \tilde g, \tilde g-\tilde f$ along $E_i$ by $a_i,b_i, c_i, \tilde a_i, \tilde b_i, \tilde c_i$, respectively.

Take any point  $P$ of  $E_i$, where $i\in\{1,\dots,m\}$. 
If $P$ is a smooth point of the total divisor, and $E_i$ in a neighborhood of $P$ 
has a local equation $x=0$, then $f=Ax^{a_i}$  
and $\tilde f=(1+u')f=\tilde Ax^{\tilde a_i}$, where $A$ and $\tilde A$ do not vanish at $P$. 
Hence $\tilde a_i=a_i$. Moreover, if $E_i$ is a component of the exceptional divisor, 
then $\tilde A|_{E_i}=A|_{E_i}$, since $u'|_{E_i}=0.$ Similarly we obtain $\tilde b_i=b_i$.

Choose $N$ big enough so that for all $i\in \{1,\dots,n\}$ 
$Nc_i> (N+1)a_i$ if $c_i>a_i$ and $N\tilde c_i> (N+1)a_i$ if $\tilde c_i>a_i$.  
This is for instance the case if $N>\max\{a_i: 1\leq i\leq n\}$.

Then the orders of meromorphic functions 
$F=(g-f)^N/f^{N+1}$ and $\tilde F=(\tilde g-\tilde f)^N/\tilde f^{N+1}$ 
along the components of the exceptional divisor are different from zero.  Hence the proper preimage 
of the curve $h=0$ (resp. $\tilde h=0$), 
which, in the complement of the total transform, is given by $F=1$ (resp. $\tilde F=1$) 
does not intersect the exceptional divisor at the smooth points of the total transform.

Take the intersection point $P$ of a component $E_i$ of the exceptional divisor with 
$E_j$, where $1\leq j\leq m$, $j\neq i$.
Choose  a local analytic coordinate system $(x,y)$ centered at $P$ such that 
$E_i$ has the equation $x=0$ and $E_j$ has the equation $y=0$. 
In these coordinates 
$f=A x^{a_i}y^{a_j}$, 
$g=B x^{b_i}y^{b_j}$, 
$g-f=Cx^{c_i}y^{c_j}$, 
$\tilde f=\tilde Ax^{a_i}y^{a_j}$, 
$\tilde g=\tilde B x^{b_i}y^{b_j}$, 
and~$\tilde g-\tilde f=\tilde C x^{\tilde c_i}y^{\tilde c_j}$,
where $A,B,C,\tilde A,\tilde B,\tilde C$ are germs of holomorphic functions that do not vanish at zero.
It follows from \cite{ZarLincei}, the first paragraph of proof of Proposition 2.1, see also \cite[Lemma 4.7]{B-M}, 
that the set of pairs $\{(a_i,a_j),(b_i,b_j),(c_i,c_j)\}$ is totally ordered, 
with respect to the partial order $(a,a') \le (b,b') $ if $a\le b$ and $a\le b'$,  
and two of these pairs are equal and are less than or equal to the third one.

In the sequel we denote by $\unit$ the germ of any holomorphic function that does 
not vanish at the origin.

Let us write the equation of 
$$h=(g-f)^N-f^{N+1}=\unit x^{Nc_i}y^{Nc_j}-\unit x^{(N+1)a_i}y^{(N+1)a_j}$$
in a neighborhood of~$P$.  

We have the following possibilities:
%---------------------------

\medskip (I) $(b_i,b_j)>(a_i,a_j)$ \par
Then $(c_i,c_j)= (\tilde c_i, \tilde c_j)  =(a_i,a_j)$ and we get 
$$h= \unit \tilde h = \unit x^{N a_i}y^{N a_j}.$$
%---------------------------

\medskip (II) $(b_i,b_j)<(a_i,a_j)$ \par
Then $(c_i,c_j)= (\tilde c_i, \tilde c_j) =(b_i,b_j)$ and we get 
$$h= \unit \tilde h = \unit x^{N b_i}y^{N b_j}.$$
%---------------------------

\medskip (III) $(b_i,b_j)=(a_i,a_j)$\par
Then $(c_i,c_j)\geq (a_i,a_j)$ and $(\tilde c_i, \tilde c_j)\geq (a_i,a_j)$ and we get
$$h= x^{Na_i}y^{Na_j}(\unit x^{N(c_i-a_i)}y^{N(c_j-a_j)}-\unit x^{a_i}y^{a_j}) $$
and similarly for $\tilde h$.
%---------------------------
 
\medskip
Write $g-f=(B-A)x^{a_i}y^{a_j}$ and $\tilde g-\tilde f=(\tilde B-\tilde A)x^{a_i}y^{a_j}$.

\medskip 
Let us consider two subcases of (III).

\smallskip
First, assume that $a_j=b_j>0$.  
 Then  $E_j$ is a component of the exceptional divisor, 
 since otherwise one of $a_j$, $b_j$ would vanish. 
 Since $A$ and $\tilde A$ are equal on the exceptional divisor, we have
 $A(0,y)=\tilde A(0,y)$ and $A(x,0)=\tilde A(x,0)$. The same equations hold for $B$ and $\tilde B$. 
 It follows that the Newton diagrams of $B-A$ and $\tilde B-\tilde A$ have the same intersection 
 points with the coordinate axes. Moreover, $B-A$ and $\tilde B-\tilde A$ are factors of $g-f$ or  $\tilde g - \tilde f$ and hence their Newton diagrams have only one vertex. 

If $(0,0)$ is the vertex of the Newton diagram $\Delta$ of $\tilde B-\tilde A$, 
then $c_i=\tilde c_i=a_i$ and $c_j=\tilde c_j=a_j$, what implies that $h=\unit x^{Na_i}y^{Na_j}$. 
If $(a,0)$ is the vertex of $ \Delta$ for some $a>0$, then $c_i=\tilde c_i>a_i$ and $c_j=\tilde c_j=a_j$, what means that $h= x^{(N+1)a_i}y^{a_j}(\unit x^{N(c_i-a_i)-a_i}-\unit y^{a_j}).$ Similarly we have in the case, when $(0,b)$ is the vertex of $\Delta$ for some $b>0$.  
The last possibility is, when the vertex of $\Delta$ is of the form $(a,b)$ for some $a,b>0$. 
In this case we have that $c_i>a_i$, $\tilde c_i>a_i$, $c_j>a_j$ and $\tilde c_j>a_j$. 
Thus $h=\unit x^{(N+1)a_i}y^{(N+1)a_j}.$

\smallskip
Next, assume that $a_j=b_j=0$. 
Then $E_j$ is a component of the proper transform of $(g-f)(\tilde g-\tilde f)=0$.  
Without loss of generality we may assume that $E_j\subset \{g-f=0\}$.
Then  $c_j>0$ and $\tilde c_j\ge0$. 
Write $B-A=\unit x^ay^{c_j}$, where $a=c_i-a_i$, and $\tilde B-\tilde A=\unit x^{\tilde a} 
y^{\tilde c_j}$, where $\tilde a=\tilde c_i- \tilde a_i$. 
Since $A|_{E_i}=\tilde A|_{E_i}$, we have $A(0,y)=\tilde A(0,y)$ and similarly 
$B(0,y)=\tilde B(0,y)$.  Therefore $a=0$ if and only if $\tilde a=0$ and if this is the case $c_j = 
\tilde c_j$.  Then $h=x^{Na_i}(\unit y^{Nc_j} - \unit x^{a_i})$ and a similar formula holds for $\tilde h$.  If $a$ and $\tilde a$ are both strictly positive then, by the assumption on $N$, $h=x^{(N+1)a_i}(\unit x^{N(c_i-a_i)-a_i}y^{Nc_j}-\unit)=\unit x^{(N+1)a_i}$. \end{proof}

\begin{Corollary}\label{C:sliding}
Let $f,g,u',u''\in\mathbb{C}\{x,y\}$ be convergent power series vanishing at zero 
such that $f$ and $g$ are coprime and let $\tilde f=(1+u')f$, $\tilde g=(1+u'')g$.
Then for positive integers $k$, $l$, $t\ne 0$ and sufficiently large integer $N$ 
the curves 
$h_t=(g^k-t f^l)^N-f^{l(N+1)}=0$ and 
 $\tilde h_t=(\tilde g^k-t \tilde f^l)^N-\tilde f^{l(N+1)}=0$
are equisingular.
\end{Corollary}
%--------------------------------------------------

\begin{proof}
It is enough to apply Lemma~\ref{L:Key} to 
$f_1=t f^l$,  $g_1=g^k$ and 
$\tilde f_1=t\tilde f^l$, $\tilde g_1=\tilde g^k$.
\end{proof}

Now, we are ready to prove the main result.

%--------------------------------------------------
\subsection{Proof of  Theorem~\ref{T:main}}\label{S:proof}
Let $D(u,v)=0$ be the discriminant  of $(f,g):(\mathbb{C}^2,0)\to(\mathbb{C}^2,0)$ 
and let $\tilde D(u,v)=0$ be the discriminant  
of $(\tilde f,\tilde g):(\mathbb{C}^2,0)\to(\mathbb{C}^2,0)$. 
Let $w=(k,l)$ be an arbitrary weight vector, where  $k$, $l$ are coprime positive integers.
Write $\ini_w D=C u^{\nu_0}v^{\nu_{n+1}}\prod_{i=1}^n(v^k-t_i u^l)^{\nu_i}$ and 
$\ini_w \tilde D=\tilde C u^{\eta_0}v^{\eta_{n+1}}\prod_{i=1}^n(v^k-t_i u^l)^{\eta_i}$. By \cite[Theorem~2.1]{Gw} 
the Newton diagrams of $D$ and $\tilde D$ coincide. Hence $\nu_0=\eta_0$, $\nu_{n+1}=\eta_{n+1}$ 
and it is enough to prove that $\nu_i=\eta_i$ for $1\leq i\leq n$.
This follows from Corollary~\ref{C:Milnor} since by Corollary~\ref{C:sliding}
for $t\neq0$ one has $\mu (h_{t})=\mu (\tilde h_{t})$.  This ends the proof. 

%------------------------------------------

\section{Corollaries}\label{S:Corollaries}
\medskip

As a direct consequence we obtain the following result.

\begin{Corollary}
Under assumptions of Theorem~\ref{T:main} the discriminants of  $(f,g)$ and $(\tilde f,\tilde g)$
have the same tangents. 
\end{Corollary}

The next corollary concerns the atypical values of the pencils $g^k-tf^l=0$ and the asymptotic critical values of the meromorphic functions  $g^k/f^l$.

\begin{Corollary}
Under the assumptions of Theorem~\ref{T:main}, for every pair of coprime positive integers 
$k,l$ we have
\begin{itemize}
\item[{\rm (i)}]  the pencils $g^k-tf^l=0$ and $\tilde g^k-t \tilde f^l=0$,
where $t\in\mathbb{C}$ is a parameter, have the same sets of atypical values.  
\item[{\rm (ii)}] the meromorphic functions  $g^k/f^l$ and 
$\tilde g^k/\tilde f^l$ have the same asymptotic critical values. 
\item[{\rm (iii)}]  the generic 
fibers of $g^k/f^l$ and $\tilde g^k/\tilde f^l$ are equisingular. 
\end{itemize}
\end{Corollary}

\begin{proof}
To verify (i) we apply Casas' formula to $H_t=v^k-tu^l$ to show that these families are $\mu$-constant for $t\ne t_i$, $i=1, \ldots, n$. 

To show (ii) let us recall after \cite[Section 5]{KMP}, that $t_0\in\mathbb{C}\cup \infty$ 
is \emph{an assymptotic critical value} of $F=f / g$ at the origin if there is a sequence of points $(x_k,y_k)\xrightarrow{ { k \to \infty }} (0,0)$, $(x_k,y_k)\ne (0,0)$, 
such that $\|(x_k,y_k)\| \|\nabla F\| \to 0$ and $F(x_k,y_k) \to t_0$. 
Therefore if $\tilde F = u F$, with $u(0)=1$, then an elementary computation shows  that the asymptotic critical values of $\tilde F$ and $F$ coincide. 

Now we show the equivalence of (i) and (ii) thus providing for both of them alternative proofs.  That is we show that the set of asymptotic critical values of $F= f/g$ coincide with the atypical values $t$ of the pencil $f-tg$.  Indeed, an elementary computation, see \cite[Proposition 5.1]{KMP}, shows that $t_0$ is the asymptotic critical value if and only if the Kuo-Verdier condition (w) fails at $(0,0,t_0)$ for the strata $(Reg X, T)$, where $Reg X$ is the regular part of $X= V(f-tg)$ and $T$ is the $t$-axis. On the other hand it is well known that for families of isolated plane curve singularities that their equisingularity is equivalent to Whitney equisingularity, see \cite[Theorem 8.1]{ZarSEII}, and in the complex domain the Kuo-Verdier condition is equivalent to Whitney's conditions, see \cite{Te3}. For a direct elementary proof that $\mu$-constant condition is equivalent to Verdier condition for the families of curve singularities $f-tg$, see \cite[Theorem 4.1]{KP2000}. 

Finally, to show (iii) we may again apply Casas' formula to the deformations 
 $s (g^k-tf^l) + (1-s)(\tilde g^k-t \tilde f^l)$ with the parameter $s\in \bC$.
\end{proof}

\begin{Corollary}\label{Cor:rescalling}
Let $(f,g)\colon (\mathbb{C}^2,0)\longrightarrow (\mathbb{C}^2,0)$ be a holomorphic 
mapping with an isolated zero. Then the initial Newton polynomial of its discriminant
is determined, up to rescalling variables, by the ideals $(f)$ and $(g)$ in $\mathbb{C}\{x,y\}$. 
\end{Corollary}

\begin{proof}
Let $u_1, u_2\in\mathbb{C}\{x,y\}$ be power series  with nonzero constant terms. 
Let $a=u_1(0,0)$ and $b=u_2(0,0)$. If $D$ (resp. $D_1$) is the discriminant of 
$(f,g)\colon (\mathbb{C}^2,0)\longrightarrow (\mathbb{C}^2,0)$ 
(resp. $(af,bg)\colon (\mathbb{C}^2,0)\longrightarrow (\mathbb{C}^2,0)$) then  
$D(au,bv)=D_1(u,v)$. %\cite[Lemma 6.1]{GaGwLe}.
Hence by Theorem~\ref{T:main} applied to $(af,bg)$ and $(u_1f,u_2g)$,
the initial Newton polynomials of the discriminants of $(f,g)$ and $(u_1f,u_2g)$ are equal 
up to rescalling variables.
\end{proof}

%\section{Application of methods}
\medskip
As an application of the methods used in this paper we present a new proof of  
\cite[Theorem 6.6]{GaGwLe}.

\begin{Theorem} [{\cite[Theorem 6.6]{GaGwLe}}] \label{tc3}
Let $h=0$ be a unitangent singular curve and let $\ell_1=0$, $\ell_2=0$ be smooth curves 
transverse to $h=0$. Then there exists a nonzero constant $d \in\bC$ such that  
the initial Newton polynomials of discriminants of mappings  
$(d\ell_1,h):(\bC^2,0)\to(\bC^2,0)$ and $(\ell_2,h):(\bC^2,0)\to(\bC^2,0)$ are 
equal. 
\end{Theorem}
%--------------------------------------------------

Keep the assumptions of Theorem~\ref{tc3} and let $d$ be the limit at the origin of the meromorphic 
function $\ell_2/\ell_1$ restricted to the tangent to the curve $h=0$. Fix positive integers $k$, $l$ and
a nonzero complex constant $t$.  Let $f=t h^l$,  $g=(d\ell_1)^k$, and  $\tilde g=\ell_2^k$. 
In order to prove Theorem~\ref{tc3} it is enough to have the following counterpart of 
Lemma~\ref{L:Key}.

\begin{Lemma} \label{L:Key1}
For sufficiently large integer $N$ the curves 
$(g-f)^N-f^{N+1}=0$ and $(\tilde g- f)^N-f^{N+1}=0$ are equisingular.
\end{Lemma}

\begin{proof}
Let $\sigma: M\to(\bC^2,0)$ be the blowing-up of $\bC^2$ at the origin.
Then the proper transforms of the curves $f=0$, $g=0$, $\tilde g=0$ 
intersect the exceptional divisor $E$ at points $P$, $Q$, $\tilde Q$ respectively
(if $g=0$ and $\tilde g=0$ have the same tangent, then $Q=\tilde Q$). 

Let $R:M_1 \to M$ be a resolution of singularities of the curve $fg(g-f)(\tilde g- f)=0$ 
at $P$. By Zariski equisingularity criterion it is enough to prove that 
the total and the proper transforms of the curves $(g-f)^N-f^{N+1}=0$ and $(\tilde g-f)^N- f^{N+1}=0$ 
by $R\circ\sigma$ are equisingular.  To prove the equisigularity on $R^{-1}(P)\setminus E$ 
it is enough to use the arguments from the proof of Lemma~\ref{L:Key} 
since the meromorphic function $\tilde g/g$ is constant and equal to 1 on the set $R^{-1}(P)$.
Thus it remains to show that the equisingularity classes of these curves are the same 
on the component $E$ of the exceptional divisor. 

The point $P$ is the intersection of $E$ and another component $E'$ of the exceptional divisor 
of $R\circ\sigma$. Choose a local analytic coordinate system $(x,y)$ centered at $P$ such that 
$E$ has the equation $x=0$ and $E'$ has the equation $y=0$. 
In these coordinates 
$f=\unit x^{a}y^{a'}$, 
$g=\unit x^{b}y^{b'}$, 
$\tilde g= \tilde \unit x^{b}y^{b'}$, and
$g-f=\unit x^{c}y^{c'}$.
The set $\{(a,a'), (b,b'), (c,c')\}$ is again totally ordered as in the proof of Key Lemma, 
and two of these pairs are equal and are less than or equal to the third one.

First, consider the case $a=b$.  Then the meromorpic function $g/f$ restricted to $E$ is nonzero 
on $E\setminus \{P,Q\}$, has order $b'-a'$ at $P$ and has order $b$ at $Q$.  Since the 
sum of orders of a meromorphic function defined on a Riemann manifold at zeroes and poles must be zero, we get $b+(b'-a')=0$ which implies that $b'<a'$.  It follows that $(a,a')\neq (b,b')$. 

If $(a,a') < (b,b')$ then $(c,c')=(a,a')$ and we get  
$$ (g-f)^N-f^{N+1} = \unit x^{N a}y^{N a'}.$$

If $(b,b') < (a,a')$ then $(c,c')=(b,b')$ and we get  
$$ (g-f)^N-f^{N+1} = \unit x^{N b}y^{N b'}.$$

Now, we will determine the class of equisingularity of $(g-f)^N-f^{N+1}=0$
at points of $E$ different from $P$ and $Q$. 
Choose a local analytic coordinate system $(x,y)$ centered at one of such a point such that
$E$ has the equation $x=0$. Then $f=\unit x^{a}$ and $g=\unit x^{b}$.

If $a<b$ then  
$$ (g-f)^N-f^{N+1} = \unit x^{N a}.$$

If $b<a$ then 
$$ (g-f)^N-f^{N+1} = \unit x^{N b}.$$

If $a=b$ then the meromorphic function $F=g/f$ restricted to $E$ has exactly one zero of order $b$
at $Q$ and exactly one pole at $P$. Hence every nonzero complex number, in particular 1, 
is a regular value of $F|_E$.  As a consequence the set $(F|_E)^{-1}(1)$ consists of $b$ points $P_i$, $1\leq i\leq b$ and at each of these points the curve $F=1$ intersects $E$ transversally. 
Thus, for every $i\in\{1,\dots b\}$ we may find a local analytic coordinate system $(x,y)$ centered 
at $P_i$ such that $g-f=x^{a}y$.  We get  in the neighborhood of $P_i$ a local equation
$$ (g-f)^N-f^{N+1} = x^{N a}(y^N-\unit x^a). $$

Finally, we will determine the equisingularity class of $(g-f)^N-f^{N+1}=0$ at $Q$. 
Choose a local analytic coordinate system $(x,y)$ centered at $Q$ such that $f=x^a$
and $g=x^b y^b$ in the neighborhood of $Q$. 

If $a\leq b$ then 
$$ (g-f)^N-f^{N+1} = \unit x^{N a}. $$

If $a>b$ then 
$$ (g-f)^N-f^{N+1} = x^{N b}[(y^b-x^{a-b})^N - x^{(a-b)N+a}] $$ and its equisingularity type is uniquely determined by $a,b,N$.\end{proof}
%------------------------------------------

\begin{Corollary}\label{Cor:rescalling2}
Let $f=0$ be a unitangent singular curve and let $\ell=0$ be a smooth curve
transverse to $f=0$. Then the initial Newton polynomial of the discriminant of the mapping
$(\ell,f):(\bC^2,0)\to(\bC^2,0)$ is determined, up to rescalling variables by the ideal 
$(f)\subset\mathbb{C}\{x,y\}$.
\end{Corollary}

\medskip
\noindent
{\small   Beata Gryszka\\
Institute of Mathematics\\
Pedagogical University of Cracow\\
Podchor\c{a}\.{z}ych 2\\
PL-30-084 Cracow, Poland\\
e-mail: bhejmej1f@gmail.com}

\medskip
\noindent
{\small   Janusz Gwo\'zdziewicz\\
Institute of Mathematics\\
Pedagogical University of Cracow\\
Podchor\c{a}\.{z}ych 2\\
PL-30-084 Cracow, Poland\\
e-mail: janusz.gwozdziewicz@up.krakow.pl}

\medskip
\noindent
{\small   Adam Parusi\'nski\\
Universit\'e C\^ote d'Azur\\
CNRS,  LJAD, UMR 7351\\
06108 Nice, France\\
e-mail: adam.parusinski@univ-cotedazur.fr}

\end{document}